\documentclass[12pt, reqno]{amsart}
\usepackage{amsfonts}
\usepackage{amssymb}
\usepackage{amsmath}
\usepackage{amsthm}
\usepackage{pgf,tikz}
\usepackage{syntonly}
\usepackage{pstricks,subfigure,epsfig}
\usepackage{amsmath,amsfonts,amssymb,amsthm,euscript,upgreek}
\usepackage{enumerate}
\usepackage{multirow}
\usepackage{appendix}

\theoremstyle{definition}
\newtheorem{definition}{Definition}

\newtheorem{remark}[definition]{Remark}

\theoremstyle{plain}

\newtheorem{theorem}[definition]{Theorem}
\newtheorem{corollary}[definition]{Corollary}

\usepackage{rubikcube}

\usepackage
[authormarkup=none]
{changes}

\definechangesauthor[color=blue]{Stef}

\begin{document}

\title[On the $n\times n\times n $ Rubik's Cube] {On the $n\times n\times n $ Rubik's Cube}


\author[S. Bonzio]{Stefano Bonzio}
\address{Stefano Bonzio, The Czech Academy of Sciences, Prague\\
Czech Republic}
\email{stefano.bonzio@gmail.com}

\author[A.Loi]{Andrea Loi}
\address{Andrea Loi, University of Cagliari\\
Italy}
         \email{loi@unica.it}

 \author[L.Peruzzi]{Luisa Peruzzi}
\address{Luisa Peruzzi, University of Cagliari\\
Italy}
\email{luisa$\_$peruzzi@virgilio.it}

\thanks{}

\date{\today}

\begin{abstract}
We state and prove the ``first law of Cubology'', i.e. the solvability criterion, for the $n\times n\times n$ Rubik's Cube.\end{abstract}

\maketitle




\section{Introduction}

Erno Rubik invented in 1974 the most famous and appreciated puzzle game of all times that still goes under his name: the Rubik's Cube. In 1981, the Rubik's Revenge appeared on the market, meant to be a more difficult puzzle with respect to the original. In a sort of race to make the challenge harder and harder, a few years later the Professor's Cube came to life, sharing some features with both the Rubik's Cube and the Rubik's Revenge.

Over the years the Rubik's Cube attracted the attention of many mathematicians (see for example \cite{Bande82}, \cite{Joyner08}, \cite{Kosniowski81}, \cite{miller2012})  who gave a group theoretical analysis and solution to the puzzle. The enigma fascinated also computer scientists, whose interests are mainly focused in establishing the God's number, namely the least number of moves needed to solve the Cube from any arbitrary position, see for example \cite{Rokicki13}, \cite{OnGods}, \cite{Rokicki14}, \cite{Kunkle07}, \cite{Demaine2011}. Strategies strictly connected to the Rubik's Cube have been recently applied also in cryptography \cite{Volte2013}, \cite{Diaconu13}, while some physicists were inspired by the Cube as a model for complex systems of different kinds \cite{Czech2011}, \cite{Lee}.

Any \emph{Cubemaster} knows that dismantling the cube and reassembling it randomly may cause, in most of the cases, that the puzzle is not solvable anymore. Therefore a question arises naturally: under which conditions is a (scrambled) cube solvable? The answer is well known for the Rubik's Cube and can be found in literature. In 1982, Bandelow \cite{Bande82} answered this question with ``the first law of cubology'', which furnishes necessary and sufficient conditions for the solvability of the Rubik's Cube. This suggests how important the question appears to mathematicians. In a recent work \cite{Loi}, we provided an answer to the same question for the Rubik's Revenge; later soon, we realized that our approach can be fruitfully raised to a more general level. The aim of the present work is to provide the most general result for a Rubik's Cube of dimension $ n $, meaning a Rubik's Cube provided with $ n $ rotating slices. The result is addressed by splitting the disjoint cases for $ n $ being an even or an odd number (see Theorem \ref{th: first law per il cubo pari} and Theorem \ref{th: first law per il cubo dispari}).

The paper is structured as follows: in Section \ref{sec: 5x5}, we address a detailed mathematical description of the Professor's Cube and state the `first law of cubology' for it (see Theorem \ref{the: condizioni5x5}). In order to prove this result we provide the algebraic analysis of the group of the cube in Subsection \ref{subsec: studio di G5}.
The $n\times n \times n$ Cube is introduced in Section \ref{sec: nxn}: this section is divided into four subsections. In Subsections \ref{subsec: n pari} and \ref{subsec: n dispari} we state the main theorem for the case where $n$ is even and odd, respectively. 
In Subsection \ref{subsec: 6x6} we address the algebraic study for the case $n=6$ then used in Subsection \ref{subsec: proof caso n=6} to give a proof of the main statement for $n$ even.

We opt, on purpose, not to give the explicit proof of the first law of cubology for the $n\times n\times n $ Cube, as we believe it makes the reading unnecessarily more difficult. On the contrary, we described in full details the examplar cases for $n=5$ (Section \ref{sec: 5x5}) and $n=6$ (see Subsection \ref{subsec: 6x6}), illustrating how the proof technique can be extended to arbitrary $n$.

\section{Configurations of the Professor's Cube}\label{sec: 5x5}

The Professor's Cube is an extension of the Rubik's Cube and the Rubik's Revenge. It is made of five rotating slices, from which it follows that the Professor's Cube is composed by 98 cubies: 8 corner cubies (possessing 3 stickers each), 36 edge cubies (2 stickers) and 54 remaining center cubies (one sticker only).
At first glance the Professor's Cube turns out to share a remarkable feature with the Rubik's Cube: one cubie in each face, namely the most central one, is \emph{fixed}. This represents a big difference with respect to the Rubik's Revenge and any cube with an even number of slices, where each center cubie can be moved. 
Furthermore, the number of edge cubies is exactly the sum of the 24 edges (twelve pairs, in particular) of the Rubik's Revenge and the 12 edges of the Rubik's Cube: we will refer to the formers as \textit{coupled edges} (indicated by black spots in Figure \ref{fig: edge singoli e in coppia}), while to the latters as \textit{single edges} (red spots in Figure \ref{fig: edge singoli e in coppia}). 

 \begin{center}
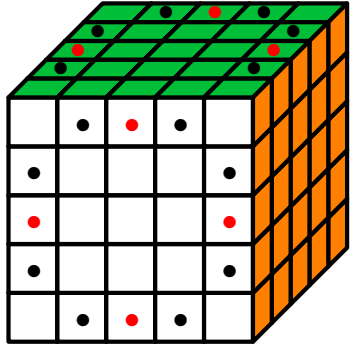
\begin{figure}[h]
\RubikCubeSolved
\begin{minipage}{6cm}
\centering
\begin{tikzpicture}[scale=0.65]
\DrawNCubeAll{5}{O}{G}{W}

\node at (-0.4, 2.5) {$\bullet$};
\node at (0.6, 2.5) {\textcolor{red}{$\bullet$}};
\node at (1.6, 2.5) {$\bullet$};
\node at (-0.4, -1.5) {$\bullet$};
\node at (0.6, -1.5) {\textcolor{red}{$\bullet$}};
\node at (1.6, -1.5) {$\bullet$};
\node at (-0.855, 3.65) {$\bullet$};
\node at (-0.5, 4.02) {\textcolor{red}{$\bullet$}};
\node at (-0.1, 4.42) {$\bullet$};

\node at (-1.4, 1.5) {$\bullet$};
\node at (-1.4, 0.5) {\textcolor{red}{$\bullet$}};
\node at (-1.4, -0.5) {$\bullet$};
\node at (1.3, 4.8) {$\bullet$};
\node at (2.3,4.8) {\textcolor{red}{$\bullet$}};
\node at (3.3, 4.8) {$\bullet$};

\node at (3.1, 3.65) {$\bullet$};
\node at (3.5, 4.02) {\textcolor{red}{$\bullet$}};
\node at (3.9, 4.42) {$\bullet$};

\node at (2.6, 1.5) {$\bullet$};
\node at (2.6, 0.5) {\textcolor{red}{$\bullet$}};
\node at (2.6, -0.5) {$\bullet$};

\end{tikzpicture}
\end{minipage}
\caption{Black and red dots indicating some coupled and singles edges, respectively.}\label{fig: edge singoli e in coppia}
\end{figure}
\end{center}

Central cubies are 9 on each face: 1 fixed and 8 moving ones. We can split the 8 moving ones into two classes: \textit{center corners}, namely the ones standing on the diagonals of a fixed center piece (blue spots in Figure \ref{fig: central corners and edges}), and \textit{center edges}, standing on the side of the fixed center piece (indicated by red spots in Figure \ref{fig: central corners and edges}). 
\begin{center}
\begin{figure}[h]
\RubikCubeSolved
\begin{minipage}{6cm}
\centering
\begin{tikzpicture}[scale=0.65]
\DrawNCubeAll{5}{O}{G}{W}

\node at (0.6,0.5) {$\bullet$};
\node at (-0.4, -0.5) {\textcolor{blue}{$\bullet$}};
\node at (0.6, -0.5) {\textcolor{red}{$\bullet$}};
\node at (1.6, -0.5) {\textcolor{blue}{$\bullet$}};
\node at (-0.4, 1.5) {\textcolor{blue}{$\bullet$}};
\node at (0.6, 1.5) {\textcolor{red}{$\bullet$}};
\node at (1.6, 1.5) {\textcolor{blue}{$\bullet$}};

\node at (1.6, 0.5) {\textcolor{red}{$\bullet$}};
\node at (-0.4, 0.5) {\textcolor{red}{$\bullet$}};

\node at (0.1, 3.65) {\textcolor{blue}{$\bullet$}};
\node at (0.5, 4.02) {\textcolor{red}{$\bullet$}};
\node at (0.9, 4.42) {\textcolor{blue}{$\bullet$}};

\node at (2.1, 3.65) {\textcolor{blue}{$\bullet$}};
\node at (2.5, 4.02) {\textcolor{red}{$\bullet$}};
\node at (2.9, 4.42) {\textcolor{blue}{$\bullet$}};

\node at (1.5, 4.02) {$\bullet$};

\node at (1.1, 3.65) {\textcolor{red}{$\bullet$}};
\node at (1.9, 4.42) {\textcolor{red}{$\bullet$}};
\end{tikzpicture}
\end{minipage}
\caption{Black spots indicating the fixed center cubies in the up and front face, red spots center edges and blue center corners.}\label{fig: central corners and edges}
\end{figure}
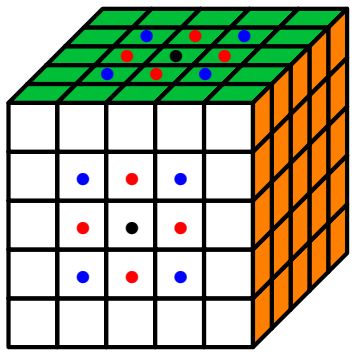
\end{center}

The colour that every face shall assume once the cube is solved is determined by the colour of the fixed center cubie living in each face. As a matter of convention, throughout the paper we mean the Professor's Cube oriented so to have the green face on top and the white one in front.

The set of moves of the Professor's cube naturally inherits the structure of a group, which we denote by $ \mathbf{M} $. The group is generated by the twelve clockwise rotations of slices denoted by capital letters $ R,L,F,B,U,D, C_{R}, C_{F}, C_{U}, C_{L}, C_{B}, C_{D} $, where $ R,L,F,B,U,D $ stand for twists of the external slices (right, left, front, back, up and down face respectively), while $ C_{F}, C_{R}, C_{U}, C_{L}, C_{B}, C_{D}  $ denote the twists of the central-front, central-right, central-up, central-left, central-back and central-down slice respectively (red arrows in Figure \ref{fig: alcune mosse del 5x5} indicates $C_L$ and $C_R$).
\begin{center}
\begin{figure}[ht]
\RubikCubeSolved
\begin{minipage}{6cm}
\centering
\begin{tikzpicture}[scale=0.65]
\DrawNCubeAll{5}{O}{G}{W}
\draw[line width= 2pt, red]  (1.6, 0.7) -- (1.6, 3.05);
\draw[line width= 2pt, ->, red] (1.6,3) -- (2.5,4);
\node at (1.6, -0.5) {\textcolor{red}{$\mathbf{C_R}$}};
\node at (2.6, -0.5) {\textcolor{blue}{$\mathbf{R}$}};

\draw[line width= 2pt, blue]  (2.6, 0.7) -- (2.6, 3.05);
\draw[line width= 2pt, ->, blue] (2.6,3) -- (3.5,4);

\draw[line width= 2pt, <-, blue]  (-1.4, 0.5) -- (-1.4, 3.05);
\draw[line width= 2pt, blue] (-1.4,3) -- (-0.5,4);

\draw[line width= 2pt, <-, red]  (-0.4, 0.5) -- (-0.4, 3.05);
\draw[line width= 2pt, red] (-0.4,3) -- (0.5,4);

\node at (-0.4, -0.5) {\textcolor{red}{$\mathbf{C_L}$}};
\node at (-1.4, -0.5) {\textcolor{blue}{$\mathbf{L}$}};

\end{tikzpicture}
\end{minipage}
\caption{Blue arrows indicate the rotation of the left and right face, i.e. the action of the moves $L$ and $R$; red indicated the rotation of the internal slices corresponding to $C_L$ and $C_R$.}\label{fig: alcune mosse del 5x5}
\end{figure}
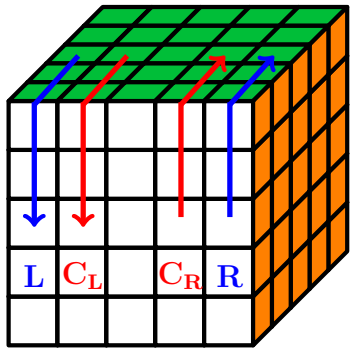
\end{center}


The number of stickers making up the Cube is equal to 150, hence we may define the group homomorphism
$$ \varphi :\mathbf{M}\longrightarrow \mathbf{S}_{150}, $$ 
\noindent
sending a move $ m\in\mathbf{M} $ to a permutation $ \varphi(m)\in\mathbf{S}_{150} $, which corresponds to the permutation in $ \mathbf{S}_{150} $ \textit{induced} by $ m $. 
\begin{remark}\label{inclusione stretta}\rm
\textrm{The inclusion $ \varphi (\mathbf{M}) \subset \mathbf{S}_{150}$ is (obviously) strict. For example, it may never happen that an element of 
$\mathbf{M}$ would send a corner to the position occupied by a center or by an edge cubie.}    
\end{remark} 

The group of the Professor's Cube is defined as $ \mathbf{G}_5 :=\varphi(\mathbf{M}) $. In other terms, it is the quotient group $ \mathbf{M}/ker(\varphi) $, i.e. it is formed by the moves we obtain by identifying all the combination of moves leading to the identical permutation.

Consider now the subset of $\mathbf{S}_{150}$ corresponding to permutations and/or orientation changes of corners, edges and center cubies. The set all of these permutations will be called
the \textit{space of configurations of the Professor's Cube} and will be denoted by $\mathcal{S}_{Conf}$\footnote{In this definition we suppose that we cannot switch for instance a single edge with a coupled edge or a center edge with a center corner.}.

\begin{remark}\label{re:no flip}
A configuration is nothing but a \emph {randomly assembled} Professor's Cube where we also allow edge flips. It is known that in the Rubik's Revenge a single edge cubie can not be flipped\footnote{This statement is proved in \cite{Larsen85}.}. In the Professor's Cube an analogous statement holds for coupled edges only and not for every edge in general (single edges actually flips!). However we may theoretically think to flip a single edge (and hence changing its orientation only) by swapping its stickers.
 Therefore, unlike the Rubik's Cube, the cardinality of $\mathcal{S}_{Conf}$ is larger than the number of patterns one may get by dismantling and reassembling the cube.  

\end{remark}

Clearly $ \mathbf{G}_5 \subset \mathcal{S}_{Conf}\subset \mathbf{S}_{150}$ and $ |\mathcal{S}_{Conf}| < 150! $. More precisely 
\begin{equation}\label{cardconf}
|\mathcal{S}_{Conf}|= (24!)^3\cdot 2^{36}\cdot 12! \cdot 3^{8}\cdot 8!
\end{equation}
We let the group $ \mathbf{G}_5 $ act on the left on $ \mathcal{S}_{Conf} $:
$$ \mathbf{G}_5 \times\mathcal{S}_{Conf}\longrightarrow\mathcal{S}_{Conf} $$
$$ (g,s)\longmapsto g\cdot s $$
where ``$ \cdot $'' stands for the composition in $ \mathbf{S}_{150} $. This gives raise to a left action of $ \mathbf{M} $ on the space of configurations, by $ m\cdot s=g\cdot s $, where $ g=\varphi(m) $, and viceversa. For this reason, from now on, we will not make any distinction between the two actions on $ \mathcal{S}_{Conf} $. Notice that  the action of $\mathbf{G}_5 $ on $ \mathcal{S}_{Conf} $ is \textit{free} (in contrast with that of $\mathbf{M}$), i.e.  if $g\cdot s = s$ then $g=id $. Hence this action yields a bijection between the group $ \mathbf{G}_5 $ and the orbit $ \mathbf{G}_5 \cdot s=\{g\cdot s\ | \ g\in\mathbf{G}_5 \} $ of an arbitrary $s\in\mathcal{S}_{Conf} $, obtained by sending $g\in\mathbf{G}_5 $ into $g\cdot s\in\mathcal{S}_{Conf} $.

\begin{remark}
It is easy to see that the space of configurations $\mathcal{S}_{Conf} $ is a subgroup of $\mathbf{S}_{150}$ containing $\mathbf{G}_5$ as a subgroup. Then the left action $g\cdot s$, $g\in \mathbf{G}_5$ and $s\in \mathcal{S}_{Conf}$, can be also seen as the multiplication in $\mathcal{S}_{Conf}$ and 
the orbit  $\mathbf{G}_5\cdot s$ of $s\in \mathcal{S}_{Conf}$ is nothing but the right coset of $\mathbf{G}_5$ in $\mathcal{S}_{Conf}$ with respect to $s$.
\end{remark}
Some pieces in the Professor's Cube, namely corners, single edges and (fixed) centers, are univocally identified by the colour of their stickers. On the other hand, ambiguity may arise concerning coupled edges, center corners and center edges. For this reason and in order to characterize mathematically the notion of \textit{configuration}, all center edges, center corners and coupled edges shall be labelled: a number between 1 and 24 may work for center corners as well as center edges\footnote{Notice that a center edge can never assume the position of a center corner and viceversa.}. Once all center cubies have been marked, the position of any of them, in a random pattern, can be described by a permutation, $ \rho_c\in S_{24}$ for center corners and $ \rho_e\in S_{24} $ for center edges.

The description of corners works exactly like for the Rubik's Cube\footnote{A detailed description can be found for instance in \cite{Chen} or in \cite{Bande82}.}: a permutation $ \sigma\in S_{8} $ describes their positions and vectors $ x\in(\mathbb{Z}_{3})^{8}$ do the same for orientations.   

We may think of the single edges as the edge pieces of the Rubik's Cube, hence their position is described by a permutation $ \tau\in S_{12} $ and the orientation by a vector $ z\in (\mathbb{Z}_{2})^{12} $. On the other hand, the twenty-four coupled edges can be divided in twelve pairs, namely those ones of the same colour. The two members of a pair are labelled with different letters: \textit{a} and \textit{b}, respectively. This is enough to provide a description of edges' positions by using a permutation $ \tau_{_1}\in S_{24} $. We refer to an edge labelled with \textit{a} (respectively \textit{b}) as an edge of \textit{type a} (respectively \textit{type b}). Obviously the type of a piece depends on its label and not on the position it is lying in. 
 
 \begin{center}

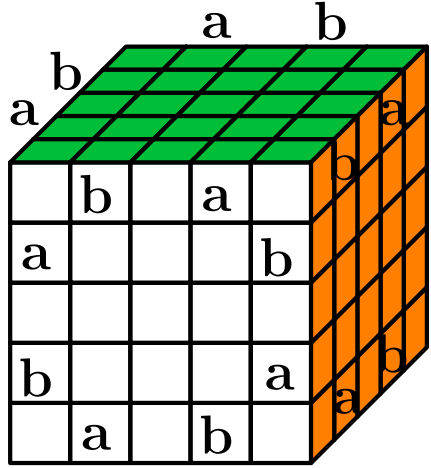
\RubikCubeSolved
\begin{minipage}{6cm}
\centering
\begin{tikzpicture}[scale=0.8]
\DrawNCubeAll{5}{O}{G}{W}

\node at (-0.5, 2.55)
{\LARGE{\textbf{b}}};
\node at (1.5, 2.47)
{\LARGE{\textbf{a}}};

\node at (-0.5, -1.5)
{\LARGE{\textbf{a}}};
\node at (1.5, -1.45)
{\LARGE{\textbf{b}}};

\node at (-1, 4.6)
{\LARGE{\textbf{b}}};
\node at (-1.7, 3.9)
{\LARGE{\textbf{a}}};

\node at (3.4, 5.43)
{\LARGE{\textbf{b}}};
\node at (1.5, 5.35)
{\LARGE{\textbf{a}}};

\node at (4.45, 3.9)
{\LARGE{\textbf{a}}};
\node at (3.6, 3.1)
{\LARGE{\textbf{b}}};

\node at (4.41, -0.1)
{\LARGE{\textbf{b}}};
\node at (3.68, -0.9)
{\LARGE{\textbf{a}}};

\node at (2.5, 1.5)
{\LARGE{\textbf{b}}};
\node at (2.55, -0.5)
{\LARGE{\textbf{a}}};

\node at (-1.5, 1.5)
{\LARGE{\textbf{a}}};
\node at (-1.5, -0.5)
{\LARGE{\textbf{b}}};

\end{tikzpicture}
\end{minipage}

\end{center}
 
\vspace{15pt}
For describing orientations of coupled edges, we proceed as we did for the Rubik's Revenge \cite{Loi}, using only twelve numbers (instead of 24) and labels \textit{a} and \textit{b}.  \\

In virtue of the labelling system introduced above, we describe the orientations of coupled edges by two 12-tuples $ y_a=(y_{1_a},y_{2_a},...,y_{{12}_a}) $, with $ y_{i_a}\in\mathbb{Z}_{2} $ and $ y_b=(y_{1_b},y_{2_b},...,y_{{12}_b}) $, with $ y_{i_b}\in\mathbb{Z}_{2} $ (edges are twenty-four, divided in pairs a and b).\\
It follows that the space of configurations $\mathcal{S}_{Conf}$  is in bijection with  the set of $8$-tuples of the form $ (\sigma , \tau , \tau_{_1}, \rho_c , \rho_e, x , y, z)$, where $ \sigma \in S_{8} $, $ \tau\in S_{12} $ , $ \tau_{_1}\in S_{24} $, $ \rho_c\in S_{24} $, $ \rho_e\in S_{24} $ while $ x\in (\mathbb{Z}_{3})^{8} $, $ y\in(\mathbb{Z}_{2})^{24} $ and $ z\in(\mathbb{Z}_{2})^{12} $.  
From now on we identify $\mathcal{S}_{Conf}$ with such $8$-tuples. 
\vspace{5pt}

The $8$-tuple $(id_{S_{8}},id_{S_{12}}, id_{S_{24}},id_{S_{24}}, id_{S_{24}}, 0 ,0, 0)$
will be called the \textit{initial configuration}.

\begin{definition}\label{conf Revenge valida}
A configuration of the Professor's Cube is \emph{valid} when it is in the orbit of the initial configuration under the action of $ \mathbf{G}_5 $.
\end{definition}

Before stating the main result of this section, it seems appropriate to make some considerations concerning coupled edges. As in the Rubik's Revenge (see \cite{Loi} for details), an edge of type \textit{a} (resp. \textit{b}) can occupy, in a random configuration, either an a-position or a b-position. It follows that, by using the information encoded in $ \tau_{_1}\in S_{24} $, we may associate to any edge a number $ i_{t,s} $, with $ t,s\in\{a,b\} $, where $ i_t $ indicates the spatial position, while $ s $ refers to the type of the edge. There always are orientation numbers associated to any edge $ i_{t,s} $ which will be $ y_{i_{t,s}}:= y_{i_t} $. 
  
We can now give the conditions for a configuration to be valid: this is actually the ``first law of cubology'' for the Professor's Cube\footnote{The present result refines its first, uncorrect, formulation in \cite{Bonziothesis}.}.
\begin{theorem}\label{the: condizioni5x5}
A configuration $ (\sigma , \tau , \tau_{_1}, \rho_c , \rho_e, x , y, z)$ of the Professor's Cube is valid if and only if 
\begin{enumerate}
\item[\emph{1.}] $sgn(\sigma)=sgn(\tau) =sgn(\rho_c) $
 \item[\emph{2.}] $sgn(\tau_1)=sgn(\sigma)sgn(\rho_e) $
\item[\emph{3.}] $ \sum_{i}x_i\equiv 0 $ \emph{mod 3} 
\item[\emph{4.}] $ \sum_{i}z_{i}\equiv 0 $ \emph{mod 2} 
\item[\emph{5.}] $ y_{i_{t,s}}=1-\delta_{t,s} $, $ i=1,...,12$,
\end{enumerate}
where  $\delta_{a,a}=\delta_{b,b}=1$ and  $\delta_{a,b}=\delta_{b,a}=0$.
\end{theorem}

Next section is devoted to the proof of this theorem. Our proof will not be constructive, meaning that we do not show the moves actually needed to solve the cube, but we will prove and consequently use some group-theoretical results.

As a corollary we get the order of $ \mathbf{G}_5 $.
\begin{corollary}\label{ordine di G}
The order of $ \mathbf{G}_5 $ is $ (24!)^{3}\cdot\; 2^{7} \cdot\; 12! \cdot\;8!\;\cdot\;3^{7} $.
\proof
We already know that the action of $ \mathbf{G}_5 $ on $ \mathcal{S}_{Conf.} $ is free. Therefore $ |\mathbf{G}_5|=|\mathbf{G}_5\cdot s| $ for all $ s\in\mathcal{S}_{Conf.} $. It follows that $ |\mathbf{G}_5|=\frac{|\mathcal{S}_{Conf.}|}{N} $, where $ N$ is the number of orbits. Theorem \ref{the: condizioni5x5} yields $N= 2^{3}\cdot 2 \cdot 2\cdot 3\cdot 2^{24}$  and the results follows by (\ref{cardconf}).
\endproof
\end{corollary}

%

\vskip 0.3cm

In order to study the solvability of the Professor's Cube we give the following:

\begin{corollary}\label{Sticker}
The probability that a randomly assembled Professor's cube is solvable is \begin{large} $\frac{1}{2^{12}\cdot 12}$ \end{large}.
\proof
In a randomly assembled Professor's Cube central pieces (both edges and corners) are not labelled, hence condition 1 and condition 2 reduce simply to $ sgn(\sigma)=sgn(\tau) $. The 24 equations in condition 4 are  reduced to 12: this is can be obtained by assigning a label $a$ or $b$ to each edge in a pair, depending on its orientation, in such a way that $y_{i_{t,s}}=1-\delta_{t,s}$.
\endproof
\end{corollary}

In \cite{Loi} we focused on the Rubik's Revenge and we observed that the Revenge sold on the market is different from the Revenge studied in the paper. The same holds for the Professor's Cube: the mathematical description acutally leads us to an object slightly different from the real one. 
In fact, as the most relevant feature of the Revenge sold on the market is that any member of a pair of edges  of the same colours is different from its companion, the same statement holds in the Professor's Cube for coupled edges. This fact has the physical effect that it is impossible to assemble the cube putting an edge of type $a$ (respectively type $b$) in a `$b$-position' (resplectively $a$-position), without changing the orientation of both edges in a pair.  This yields that condition 5 in Theorem \ref{the: condizioni5x5} can be always achieved, due to the internal mechanism of the Professor's Cube.  Thus, surprisingly enough,  we get:
\begin{corollary}\label{Monkey}
The probability that a randomly assembled Professor's Cube sold on the market  is solvable is \begin{large} $\frac{1}{12}$ \end{large}.
\end{corollary}

\subsection{On the subgroups of $\mathbf{G}_5 $}\label{subsec: studio di G5}

In this section we study of the structure of $ \mathbf{G}_5 $. In particular we aim at showing that some subgroups of $ \mathbf{G}_5 $, namely those ones that permute corners, single edges, and center corners are alternating, while those acting on coupled edges and center edges are symmetric groups. This allows an elegant, although non constructive, proof of Theorem \ref{the: condizioni5x5}. 

The significant subgroups of $ \mathbf{G}_5 $ we want to study will be denoted by
$\mathbf{C}$, which permutes corner cubies (no matter the action on orientation), and act as the identity on other pieces; $ \mathbf{E} $, permuting single edges only (and acting as identity on every other piece); $ \mathbf{E_{_c}} $, acting on coupled edges only; $ \mathbf{Z_{_c}} $ permuting center corners only and finally $ \mathbf{Z_{_e}} $ acting only on center edges. 

It is easy to notice that corners and single edges in the Professor's Cube act exactly as corners and edges of the Rubik's Cube. Indeed the corresponding subgroups permuting them only, $ \mathbf{C} $ and $ \mathbf{E} $ respectively, are exactly those of the Rubik's Cube. Since each corner may assume three different orientations, it is known \cite{Signm82}, \cite{Larsen85} that the subgroup of corners corresponds to the wreath product $ \mathbf{H} = \mathbf{S}_{8} \bigotimes_{Wr} \mathbb{Z}_{3} $. However, we aim at describing $\mathbf{C}$ which is obtained as the quotient group $ \mathbf{H}/\mathbf{T} $, where $ \mathbf{T} $ is the (normal) subgroup consisting of all possible twists. 

\begin{theorem}\label{C=A8}
$ \mathbf{C}\cong \mathcal{A}_{8} $, the alternating group of even permutations.  
\end{theorem}
\begin{theorem}\label{E=A12}
$ \mathbf{E}\cong \mathcal{A}_{12} $, the alternating group of even permutations.
\end{theorem}

The content of the above theorems is a well known fact concerning the Rubik's Cube, hence we refer to \cite{Signm82} for their proofs.  

The presence of coupled edges and of many different center pieces make the Professor's Cube essentially different both from the Rubik's and the Revenge cubes. In the next theorems we make use of the commutator, formally for two $ m,n\in\mathbf{G}_5 $, 
$$ [m,n]=m\cdot n\cdot m^{-1}\cdot n^{-1}. $$ 

The are 24 center corners, hence necessarily $ \mathbf{Z_{_c}}\leqslant S_{24} $. 
\begin{theorem}\label{Zc=A24}
$ \mathbf{Z_{_c}}\cong\mathcal{A}_{24} $, the alternating group of even permutation. 
\proof
We first show that $ \mathcal{A}_{24} \leqslant \mathbf{Z_{_c}} $. The move 
\begin{equation}\label{mossa z}
z= [[C_{F},C_{D}],U^{-1}],
\end{equation}

is a 3-cycle on center corners and acts as identity on all the remaining pieces (see Figure \ref{fig: mossa z}).

\begin{center}
\begin{figure}[h]
\RubikCubeSolved
\begin{minipage}{6cm}
\centering
\begin{tikzpicture}[scale=0.65]
\DrawNCubeAll{5}{O}{G}{W}

\node at (0.15, 3.65)
{$\bullet$};
\node at (2.15, 3.65)
{$\bullet$};

\node at (-0.41, -0.499)
{$\bullet$};

\node at (10.15, 3.65)
{$\bullet$};
\node at (12.15, 3.65)
{$\bullet$};
\node at (9.59, -0.499)
{$\bullet$};

\draw[line width= 1pt, ->] (10.1, 3) -- (9.65,0);
\draw[line width= 1pt, ->] (10, -0.1) -- (11.85,3);
\draw[line width= 1pt, ->] (11.8, 3.65) -- (10.5,3.65);
%
\end{tikzpicture}
\end{minipage}
\caption{The 3-cycle action of the move $z$ on center corners.}\label{fig: mossa z}
\end{figure}
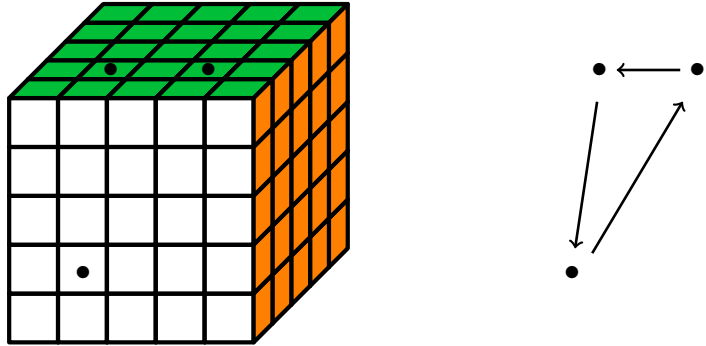
\end{center}
Observe that any three arbitrary target center corners can be moved to the positions permuted by $ z $ by a certain element $ g\in\mathbf{G}_5 $. Hence by the move $ g\cdot z\cdot g^{-1} $ we may cycle any center corners. As $ \mathcal{A}_{24} $ is generated by any 3-cycle on a set of twenty-four elements, we have the desired inclusion. 
For $ \mathbf{Z_{_c}}\leqslant\mathcal{A}_{24} $, we show that any odd permutation involving center corners permutes necessarily also some other piece, hence it cannot be in $ \mathbf{Z_{_c}} $. Indeed, suppose that there exists $ \alpha\in\mathbf{Z_{_c}} $ such that $ sgn(\alpha )=-1 $: $ \alpha $ shall be obtained as a sequence of basic moves. With no loss of generality we can assume that $ \alpha $ is a sequence of $ L,R,U,D,F,B $, since the moves $ C_{R}, C_{F}, C_{U}, C_{L}, C_{B}, C_{D} $ consist of an even permutation on center corners. On the other hand, any of the moves among $ L,R,U,D,F,B $ induces 4-cycles on center corners, center edges, corners, singles edges, respectively, and two four cycles on coupled edges. Hence the move $\alpha$ induces a move $ \beta=(\beta_{1},\beta_{2},\beta_{3}, \beta_{4})\in \mathbf{C}\times \mathbf{E}\times \mathbf{E_{_c}}\times\mathbf{Z_{_e}} $, such that $ sgn(\beta)= -1 $. But since $ sgn(\beta_{3})=+1 $, then one among $ \beta_1 $, $ \beta_2 $ and $ \beta_4 $ is different from the identity, therefore $ \alpha\not\in\mathbf{Z_{_c}} $, which gives raise to a contradiction. 
\endproof 
\end{theorem}

We now aim at studying the subgroup $ \mathbf{E_{_c}} $ of moves involving coupled edges only. Those edges are 24, each of which can assume two different orientations, however no single edge in a couple can be flipped (see Remark \ref{re:no flip}). 
\begin{theorem}\label{Ec=S24}
$ \mathbf{E_{_c}}\cong \mathcal{A}_{24} $.
\end{theorem}
\begin{proof}
We first show that $ \mathcal{A}_{24} \leqslant \mathbf{E_c} $. Indeed the move 
\begin{equation}\label{mossa e}
 e=[C^{-1}_{L},[L,U^{-1}]] 
 \end{equation}
is a 3-cycle on the coupled edges identified by the black dots in Figure \ref{fig: triciclo e}

 \begin{center}
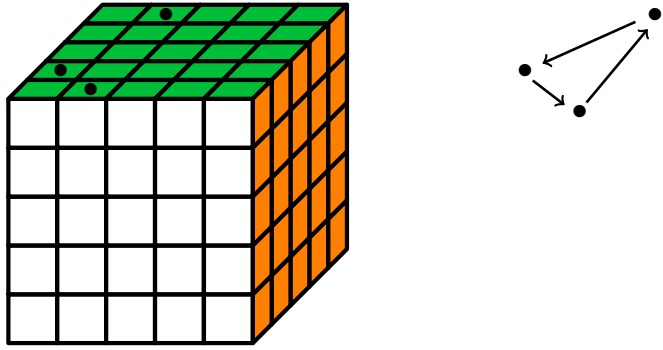
\begin{figure}[h]
\RubikCubeSolved
\begin{minipage}{6cm}
\centering
\begin{tikzpicture}[scale=0.65]
\DrawNCubeAll{5}{O}{G}{W}

\node at (-0.24, 3.25)
{$\bullet$};
\node at (-0.855, 3.65)
{$\bullet$};
\node at (1.3, 4.8)
{$\bullet$};

\node at (9.76, 2.8) {$\bullet$};
\node at (8.645, 3.65) {$\bullet$};
\node at (11.3, 4.8) {$\bullet$};
\draw[line width= 1pt, ->] (9.9, 3) -- (11.15,4.5);
\draw[line width= 1pt, ->] (10.9, 4.65) -- (9,3.8);
\draw[line width= 1pt, ->] (8.8, 3.45) -- (9.45,2.95);

\end{tikzpicture}
\end{minipage}
\caption{The 3-cycle on coupled edges moved by $e$.}\label{fig: triciclo e}
\end{figure}
\end{center}

As previously mentioned for centers, one can bring any target (coupled) edge in the positions switched by $ e $ using an element of $ g\in\mathbf{G}_5 $ and then solving the mess created by $ g^{-1} $. In this way, one obtains any 3-cycles in $ \mathbf{E_c} $, proving the desired inclusion. \\
\noindent
In order to get the inclusion  $ \mathbf{E_{_c}}\leqslant\mathcal{A}_{24} $, we  will show that any odd permutation involving coupled edges permutes necessarily also some other piece. Indeed,  suppose that there exists $ \beta\in\mathbf{E_{_c}} $ such that $ sgn(\beta )=-1 $: $ \beta $ shall be obtained as a sequence of basic moves. With no loss of generality we can assume that $\beta$ is a sequence of 
$C_{R}, C_{F}, C_{U}, C_{L}, C_{B}, C_{D}$ since the move  $ L,R,U,D,F,B $ induce an even permutation on coupled edges. On the other hand, any of the moves among $C_{R}, C_{F}, C_{U}, C_{L}, C_{B}, C_{D}$ induces 4-cycles on coupled edges and center edges and two 4-cycles on  corner edges. Thus  the move $\beta$ induces a move $ \gamma=(\gamma_{1},\gamma_{2})\in \mathbf{Z_{_c}}\times\mathbf{Z_{_e}} $, such that $ sgn(\gamma)= -1 $. But since $ sgn(\gamma_{1})=+1 $, then $\gamma_2$ is different from the identity and hence $\beta\not\in\mathbf{E_{_c}} $, the desired  contradiction. 
\end{proof}   
\begin{theorem}\label{Ze=S24}
$ \mathbf{Z_{_e}}\cong\mathcal{A}_{24} $.
\end{theorem}
\begin{proof}
We proceed following the same path of the proof of Theorem \ref{Zc=A24}. Indeed it is easy to check that the move 
\begin{equation}\label{mossa z}
w=[[C_{R}^{-1},C_{D}^{-1}],U],
\end{equation}
is a 3-cycle on central edges and acts as an identity on all the other pieces, as depicted in Figure \ref{fig: triciclo $w$}.

 \begin{center}
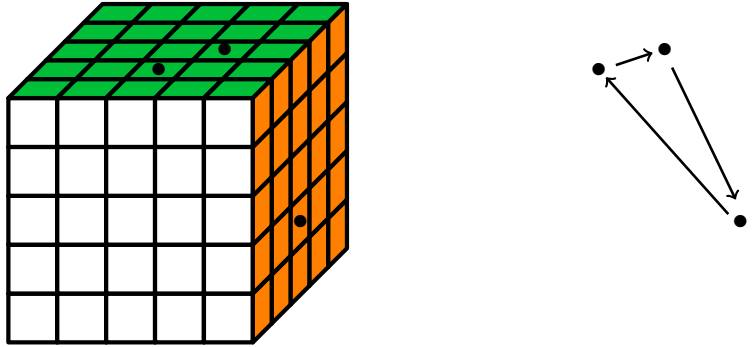
\begin{figure}[h]
\RubikCubeSolved
\begin{minipage}{6cm}
\centering
\begin{tikzpicture}[scale=0.65]
\DrawNCubeAll{5}{O}{G}{W}

\node at (1.15, 3.66)
{$\bullet$};
\node at (2.5, 4.05)
{$\bullet$};
\node at (4.05, 0.53)
{$\bullet$};

\node at (10.15, 3.66) {$\bullet$};
\node at (11.5, 4.05) {$\bullet$};
\node at (13.05, 0.53) {$\bullet$};

\draw[line width= 1pt, <-] (10.3, 3.5) -- (12.8,0.7);
\draw[line width= 1pt, <-] (12.95, 1) -- (11.65,3.7);
\draw[line width= 1pt, <-] (11.25, 4) -- (10.5,3.75);
\end{tikzpicture}
\end{minipage}
\caption{The 3-cycle induced by $w$ on center edges.}\label{fig: triciclo $w$}
\end{figure}
\end{center}

Notice once more that any three arbitrary target center edge can be placed in the positions permuted by $ w $ by a certain element $ g\in\mathbf{G}_5 $. Therefore, by the move $ g\cdot w\cdot g^{-1} $ we may cycle any center cubies. As $ \mathcal{A}_{24} $ is generated by any 3-cycle on a set of twenty-four elements, we have that $\mathcal{A}_{24} \leqslant \mathbf{Z_{_e}} $. \\
\noindent
Assume by contradiction that  $\mathbf{Z_{_e}}\leqslant \mathcal{A}_{24}$ is not satisfied, namely 
 $\mathbf{Z_{_e}}= S_{24}$. 
Consider  the move $ C_R $: it consists of an even permutation on center corners (two 4-cycles) and of an odd permutation on both coupled edges and center edges (a 4-cycle on each, respectively). Combining our assumption with Theorem \ref{Zc=A24} one can find moves $ \phi\in\mathbf{Z_e} $ and $ \varphi\in\mathbf{Z_c} $ such that $ \phi\circ\varphi\circ C_R $ acts as an odd permutation on center edges only, in constrast with Theorem \ref{Ec=S24}.
\end{proof}
\noindent
\textbf{Proof of Theorem \ref{the: condizioni5x5}} \\
$ (\Rightarrow) $ This implication is proven by checking that conditions 1,2,3,4 of Theorem \ref{the: condizioni5x5} are preserved by the basic moves. As any move is generated by them and the initial configuration trivially satisfies all the conditions above, this implies that any valid configuration does. \\

\noindent
1. and 2. We divide the basic moves into two subsets, $ M_{1}=\{ R, L, U, D, F, B\} $ and $ M_{2}=\{ C_R, C_L, C_U, C_D, C_F, C_B\} $. Moves in $ M_{1} $ consist of cycles of 4 elements each on corners, single edges and center corners and two $4$-cycles on coupled edges, hence necessarily preserve conditions $sgn(\sigma)=sgn(\tau)=sgn(\rho_{c})$ and $sgn(\tau_1)=sgn(\sigma)sgn(\rho_e) $. On the other hand, moves in $ M_2 $ act as identity on both corners and single edges, as two  $4$-cycles on center corners and as a $4$-cycle on center edges, therefore also in this case one has  $ sgn (\sigma ) = sgn (\tau)=sgn (\rho_{c}) $ and $sgn(\tau_1)=sgn(\sigma)sgn(\rho_e)$.   
\\
\\
3. $ \sum_{i}x_{i}\equiv 0 $(mod 3) follows from the fact that moves changing orientations of corners can be only generated by $ R,L,U,D,F,B $. Then corners of the Professor's Cube work exactly as those ones of the Rubik's Cube, where such a condition holds.  
\\
\\
4. $ \sum_{i}z_{i}\equiv 0 $(mod 2) is satisfied for the same reason of 2, i.e. singles edges orientation can be changes only by $ R,L,U,D,F,B $ and these moves always preserve the condition. 
\\
\\
5. First of all notice that in the initial configuration, it holds $ y_{i_t}= 0 $ for all $ i\in\{1,...,12\}$ and  $\delta_{a,a} =\delta_{b,b}= 1 $, therefore $ y_{i_{t,s}}=1-\delta_{t,s}= 0 $. 

As a valid configuration is in the orbit of the initial one, it is obtained by a sequence of basic moves, thus we need to check that those moves preserve condition $ y_{i_{t,s}}= 1 - \delta_{t,s} $. 

We consider moves splitted again in two sets (this time differently from above): $ M_{j}=\{R, U, D, L\} $ and $ M_{k}=\{F, B, C_{R}, C_{F}, C_{U}, C_{L}, C_{B}, C_{D} \} $; hence we have two possibilities: we may assume a basic move, say $ m $, either $m\in M_{j}$ or $m\in M_{k}$. \\
Assume $m\in M_{j}$. Recall that for the convention we have introduced about the assignation of orientation numbers to edges, $ m $ does not change edge cubies' orientation, so we get $ y_{i_{t,s}}= 0 $ for all $ i\in\{1,...,12\} $. Furthermore $ m $ acts on a configuration moving edges occupying an a-position in edges in a-position and the same holds for b-positions and hence $\delta_{t,s} = 1$. 

Let now $m\in M_k $. $ m $ changes orientations of some edges (the ones that it is actually permuting): more precisely it gives raise to a cycle of four edges or to two cycles of four edges each. Let $ i_{t,s} $ be one of those edges, then $ y_{i_{t,s}}= 1 $ and $ \delta_{t,s}=0 $ since a-positions and b-positions are swapped by $ m $.
\\
\\
$ (\Leftarrow ) $ 
Assume that the Professor's Cube is in a random configuration $ (\sigma, \tau, \tau_{_1}, \rho_c , \rho_e, x,y, z) $ satisfying conditions 1 to 5. We can check (simply by watching the cube) whether $ \sigma\in S_{8} $ is even or odd. If $ sgn(\sigma)= -1 $, it is enough to apply one among $ \{ R,L,U,D,F,B\} $ to get $ sgn(\sigma) = +1 $. 
Therefore in any case we can reduce to a configuration such that $sgn(\sigma)= +1 $. It follows that $ \sigma\in \mathcal{A}_{8} $, and, since by Theorem \ref{C=A8},  $ \mathbf{C}\cong\mathcal{A}_{8} $, there exists a move $ c_1\in\mathbf{C} $ such that $ c_{1}\cdot (\sigma, \tau, \tau_{_1}, \rho_c , \rho_e, x,y, z) = (id_{S_8}, \tau, \tau_{_1}, \rho_c, \rho_e,  x,y,z ) $. \\
By condition 1, in the obtained configuration, we have that $ sgn(\sigma) = sgn (\tau)= sgn (\rho)=sgn(id)=+1 $. Therefore, as $ \mathbf{E}\cong \mathcal{A}_{12} $ and $ \mathbf{Z_c}\cong\mathcal{A}_{24} $ (by Theorems \ref{E=A12} and \ref{Zc=A24}), there exist two moves, $ e\in\mathbf{E} $  and $ z\in\mathbf{Z_{_c}} $, respectively, such that  $ (e\circ z)\cdot (id_{S_8}, \tau, \tau_{_1}, \rho_c, \rho_e,  x,y,z )= (id_{S_8},id_{S_{12}} , \tau_{_1}, id_{S_{24}}, \rho_e,  x,y,z )  $. In the previous configuration, all pieces have been correctly positioned with exception of coupled edges and center-edges.

By condition 2. one has $sgn (\tau_1)=sgn (id)sgn (\rho_e)=sgn (\rho_e)$. 
Hence there are two possibilities, either $sgn (\tau_1)=sgn (\rho_e)=+1$ or $sgn (\tau_1)=sgn (\rho_e)=-1$.
If $sgn (\tau_1)=sgn (\rho_e)=-1$ by acting on the Professor 's Cube by the move  $C_R$ followed by an element of $\mathbf{Z_{_e}}=\mathcal{A}_{24}$ (by Theorem \ref{Ze=S24}) we get $sgn (\tau_1)=sgn (\rho_e)=+1$.
In this case by  Theorems \ref{Ec=S24} and \ref{Ze=S24} there exist   two moves $ f\in\mathbf{E_{_c}} $ and $ t\in\mathbf{Z_{_e}} $ such that $ (f\circ t)\cdot (id_{S_{8}},id_{S_{12}} , \tau_{_1}, id_{S_{24}}, \rho_e,  x,y,z )=  (id_{S_{8}},id_{S_{12}} , id_{S_{24}}, id_{S_{24}}, id_{S_{24}} ,  x,y,z ) $.
 
 At the present stage all cubies are correctly located, they shall only be correctly oriented. 

Condition 5 implies necessarily that whenever coupled edges are correctly located then they are also correctly orientated, namely $ y = 0 $.

It remains only to fix corners and single edges's orientations. But as they work as in the Rubik's cube, it is a well known fact that they can be always correctly oriented whenever conditions 3 and 4  are fulfilled, see for instance \cite{Bande82}.

We have proved that the initial configuration is in the orbit of a random one satisfying conditions 1 to 5.

\section{The $ n\times n\times n $ Rubik's Cube}\label{sec: nxn}

The $ n\times n\times n $ Rubik's Cube is a mathematical abstraction, with a physical counterpart for $ n\leq 17 $, to our best knowledge. By the $ n\times n\times n $, we mean an arbitrary extension of the original Rubik's Cube, possessing $ n $ (rotating) slices in every face. For $ n=3 $, it coincides with the Rubik's Cube, for $ n=4 $ with the Rubik's Revenge and for $ n= 5 $ with the Professor's Cube. 
It is not possible to describe the generalized $ n\times n\times n $ in terms of its cubies, however we aim at stating the first law of Cubology for it. 

There exists a relevant, intrinsic difference in the $ n\times n\times n$, depending on whether $ n $ is an odd or even number. In the former case, the cube presents 6 fixed center cubies that determine the colours any face shall assume through the resolution of the puzzle. On the other hand, whenever $ n $ is an even number, there is no fixed central piece (any center cubie can be rotated) and the number of center cubies will also be even.\footnote{For $ n= 4 $, for example, we have 4 center cubies in each face, for $ n=6 $, 16 in each face.}    

The above mentioned difference among `even cubes' and `odd cubes' is the leading motivation to treat the two cases separately. 

Abstracting enough, we can think of describing the moves on the generalized cube. The most external slices coincide with the faces of the cube, hence it still makes sense to call the rotation of the faces with $ R, L, U, D, F, B $ (right, left, up, down, front and back face, respectively). The number of `internal slices', meaning those slices which do not coincide with a face of the cube, depends on $ n $. \\ 
Upon noticing that central cubies are disposed in concentric circles and that a center cubie standing in each circle cannot be moved in any other circle, it makes sense to index such circles by $ k $, with $0\leq k\leq \frac{n-3}{2}$ in the `odd' case and $ 1\leq k\leq \frac{n}{2}-1 $ in the `even'.

 \begin{center}
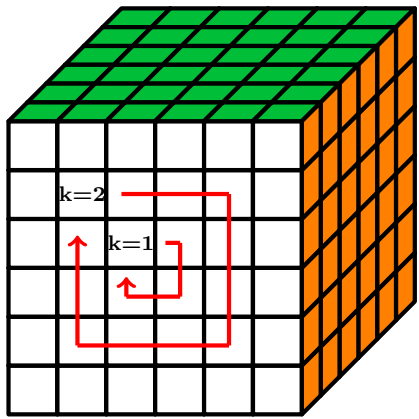
\begin{figure}[h]\label{fig:k pari}

\RubikCubeSolved
\begin{minipage}{10cm}
\centering


\begin{tikzpicture}[scale=0.65]
\DrawNCubeAll{6}{O}{G}{W}

\node at (0.2, 1.2)
{\tiny{\textbf{k=1}}};
\draw[-, ultra thick, color=red] (0.9, 1.2) -- (1.2, 1.2);
\draw[-, ultra thick, color=red] (1.2, 1.2) -- (1.2, 0.1);
\draw[-, ultra thick, color=red] (1.2, 0.1) -- (0.1, 0.1);
\draw[->, ultra thick, color=red] (0.1, 0.1) -- (0.1, 0.5);

\node at (-0.8, 2.2)
{\tiny{\textbf{k=2}}};
\draw[-, ultra thick, color=red] (0, 2.2) -- (2.2, 2.2);
\draw[ultra thick, -, color=red] (2.2, 2.2) -- (2.2, -0.9);
\draw[ultra thick, -, color=red] (2.2, -0.9) -- (-0.9, -0.9);
\draw[ultra thick, ->, color=red] (-0.9, -0.9) -- (-0.9, 1.35);

\end{tikzpicture}
\end{minipage}

\caption{Enumeration for concentric circles in the even case $n=6$.}
\end{figure}
\end{center}

\begin{center}
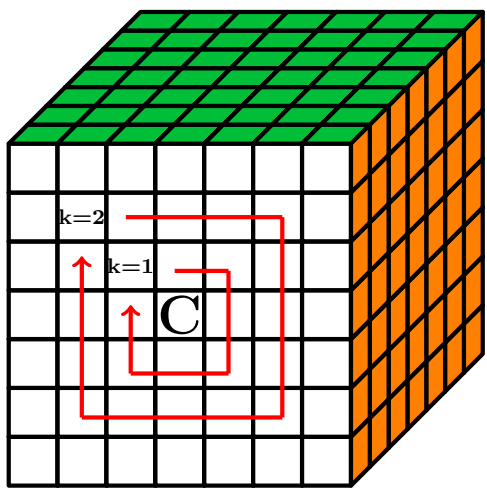
\begin{figure}[h]\label{fig:k dispari}

\begin{tikzpicture}[scale=0.65]
\DrawNCubeAll{7}{O}{G}{W}

\node at (0.8, 0.8)
{\LARGE{\textbf{C}}};

\node at (-0.2, 1.8)
{\tiny{\textbf{k=1}}};
\draw[-, ultra thick, color=red] (0.7,1.7) -- (1.8, 1.7);
\draw[ultra thick, -, color=red] (1.8,1.7) -- (1.8, -0.4);
\draw[-, ultra thick, color=red] (1.8, -0.4) -- (-0.2, -0.4);
\draw[->, ultra thick, color=red] (-0.2, -0.4) -- (-0.2, 1);

\node at (-1.2, 2.8)
{\tiny{\textbf{k=2}}};
\draw[-, ultra thick, color=red] (-0.3, 2.8) -- (2.9, 2.8);
\draw[ultra thick, -, color=red] (2.9, 2.8) -- (2.9, -1.3);
\draw[ultra thick, -, color=red] (2.9, -1.3) -- (-1.2, -1.3);
\draw[ultra thick, ->, color=red] (-1.2, -1.3) -- (-1.2, 2);

\end{tikzpicture}

\caption{Enumeration for concentric circles in the odd case $n=7$: $\textbf{C}$ indicates the fixed center.}
\end{figure}
\end{center}

With $ k = 1 $ - in the `even' Cube - we refer to the most internal circle (built of 4 cubies in each face), while the index $ k= 0 $ refers to the limit case of a degenerate circle made by the unique fixed central cubie, in the `odd' Cube.

Any internal slice stands closer to exactly an external face. For this reason, we can adopt the same notation from the previous section and indicating with $ C_{f,k} $, where $ f\in\{ R, L, U, D, F, B \} $, the rotation of the internal slice (relatively) close to the face $ f $, with $ k $ indexing the circle the move is acting on. 
As a matter of convention, from now on we write $ C_f $ instead of $ C_{f,1} $. 

\subsection{The even $n\times n \times n $ Cube}\label{subsec: n pari}

The `even cube' always 8 corner cubies (3 possible orientations), while the number of edges and centers depends on $ n $. It is not difficult to check that the numbers of centers as well as that of edges is even. In particular, the number of center cubies is 
\begin{equation}\label{eq: numero dei centri caso pari}
c = 6(n-2)^{2}, 
\end{equation}
while the number of edges is
\begin{equation}\label{eq: numero di edge}
e= 12 (n - 2). 
\end{equation}

In the special case of the Rubik's Revenge ($ n= 4 $), we have $ c = e = 24 $. Notice that in the limit case $n=2 $ we get a cube with no edges nor centers (only corners!), which is exactly the case of the $ 2\times 2\times 2 $, the so-called Pocket Cube.

The number of edges, counted in \eqref{eq: numero di edge}, is even: all of them are coupled edges.

Positions and orientations of corners can be mathematically described by a permutation $ \sigma\in S_{8} $ and a vector $ x\in (\mathbb{Z}_{3})^{8} $, respectively.

Since any center cubie can be moved, a permutation in $ S_{c} $ - with $ c $ defined in \eqref{eq: numero dei centri caso pari} - is, in principle, enough to describe the position of centers. However a more refined description of center cubies is required for our study which shall necessarily be based on concentric circles the pieces belong to.
It happens that cubies standing in any circle, with exception of the most internal one ($ k= 1 $), can be splitted into center corners and center edges. The former are the four cubies (in each circle) situated on the diagonals (with respect to the most internal circle) while the latter stand close to the side of the most internal circle.

%
%
%
%
%
%
%
%
%
%
%
%

Due to the mechanical structure of the Cube, center edges cannot be moved into the spatial position occupied by the center corners and viceversa. 
Notice that, in any circle, the number of center corners is equal to 24; hence we describe the position assumed by any of them by a permutation $ \rho_{c_{k}}\in S_{24} $, where index refers to a specific circle $k $. 

The total number of center edge cubies, which is $ 24(n-4) $ and the number $ z_k $ of center edges living in the $ k $-th circle, with $2\leq k$, is equal to $ 48(k-1) $. For this reason, a permutation $ \rho_{e_{k}}\in S_{z_{k}} $ describes the position assumed by center edge cubies. 
From now on, let us refer to the total number of circles as $K$; clearly, in the `even case', $K=\frac{n}{2}-1$.
Any edge in the `even cube' is part of a pair strictly connected to a certain circle $k$: for instance, an edge moved by $ C_{R,k} $ will be in pair with one of the edges (possibly) moved by $ C_{L,k} $, for any $ k $. The same holds for $ C_{U,k} $ with respect to $ C_{D,k} $ and for $ C_{F,k} $ with $ C_{B,k} $. For this reason, the position of edges is a concatenation of permutations $\tau_1, \tau_2,...,\tau_K $, acting on each single layer, hence it is described by a permutation $\tau\in (S_{24})^{K}$. Furthermore, taking up the same notational convention introduced  above for coupled edges in the Professor Cube, we assume that, in any pair of edges, one is of type $ a $ and its companion is of type $ b $. Clearly, also spatial positions are divided into type $ a $ and $ b $, remembering the concept of type is used both to describe a property of (edge) cubies and of (spatial) positions.
A given spatial position occupied by an edge cubie  is described recurring both to a  number  $1\leq k\leq K$, which specifies the ``circle'' it is naturally connected to (see the above discussion)
and to a number $j_t$ ($j=1, \dots 12$, $t\in \{a, b\}$), as edges connected to each circles are 24, divided in twelve pairs. Thus  a spatial position is determined by the symbol $k_{j_t}$.
Moreover, to a given edge cubie, we can associate its spatial position and its type. We write $k_{j_{t, s}}$ to denote an edge cubie of type $s\in \{a, b\}$
occupying the spatial position $k_{j_t}$.

Orientations of edges can be described recurring to a vector $ y\in(\mathbb{Z}_{2})^{e} $. More precisely, $y$ is just a concatenation of vectors $y_1, y_2,...,y_K $, each of which corresponds to a coupled edge configuration for layers $1, 2,...,K$, respectively. Thus,  $y=y_k\in ((\mathbb{Z}_{2})^{24})^{K}$.

A configuration for the `even' Cube is just a tuple $ (\sigma, \tau, \rho_{c_{k}}, \rho_{e_{k}}, x, y_k) $, with $ \sigma\in S_{8} $, $\tau\in (S_{24})^{K}$, $ \rho_{c_{k}}\in S_{24} $, $ \rho_{e_{k}}\in S_{z_{k}} $, for each $ k $ and $ x\in (\mathbb{Z}_{3})^{8} $, $y\in ((\mathbb{Z}_{2})^{24})^{K}$.

The initial configuration is the tuple where all permutations are the respective identity and any vector's component is equal to 0. A configuration is valid if and only if it stands in the orbit of the initial configuration, i.e. it might be obtained by the application of a finite sequence of moves from the initial configuration.  

We aim at stating the first law of cubology for the `even cube' which establishes the validity of an arbitrary configuration. 

\begin{theorem}\label{th: first law per il cubo pari}
A configuration $ (\sigma, \tau, \rho_{c_{k}}, \rho_{e_{k}}, x, y_k) $ of the `even Cube' is valid if and only if
\begin{enumerate}
\item[\emph{1.}] $ sgn (\sigma)=sgn (\rho_{c_{k}}) $, for every $1\leq k\leq K$,
\item[\emph{2.}] $ \sum_{i} x_{i}\equiv 0  \;\mathrm{mod}\; 3$,
\item[\emph{3.}] $y_{k_{ j_{t,s}}} = 1 - \delta_{t,s} $,
where $\delta_{a,a}=\delta_{b,b}=1$,  $\delta_{a,b}=\delta_{b,a}=0$ and $y_{k_{ j_{t, s}}}$ is the orientation of the edge cubie $k_{ j_{t, s}}$ with  
$j=1,...,12$, $t,s\in\{ a,b\}$,
\item[\emph{4.}] $sgn (\rho_{e_{k}})=+1$, for every $2\leq k\leq K$.
\end{enumerate}
\end{theorem}

\begin{remark}\label{rem: il Revenge segue dal teorema generale con n pari}
The `first law of Cubology' for the Rubik's Revenge proved in \cite{Loi} is nothing but a special case (for $ n = 4 $ and $ k=1 $) of Theorem \ref{th: first law per il cubo pari}. In such case there are no center edges and only twentyfour center corners (the only center cubies there), hence condition (4) is trivially satisfied. 
\end{remark}

\subsection{The complete study of the case $n=6$}\label{subsec: 6x6}

The number $ g $ of stickers one may find on the cube depends on $n$, namely 
\begin{equation}\label{eq: n° sticker al variare di n}
g = 6 n^{2}.
\end{equation}

The cube under analysis ($n=6$) presents 216 stickers. We refer to $\mathbf{M}_6$ as the group of its moves. We moreover define the group of the $6\times 6 \times 6$ cube by the group homomorphism 
\begin{equation}\label{eq: definizione phi}
 \varphi: \mathbf{M}_{6}\rightarrow S_{216},
 \end{equation}
sending a move $ m\in\mathbf{M}_6 $ to the corresponding permutation of stickers induced by it. The group of the Cube, $ \mathbf{G}_6 $, is the quotient group $\mathbf{M}_6/ker(\varphi) $.
We focus on certain subgroups of $ \mathbf{G}_6 $ acting on cubies of the same kind only (for example corners, edges, centers and so on) and leaving all the others untouched.  

Notice that corners behave exactly as in the Rubik's Cube. For this reason, the subgroup  acting on corners only corresponds to a subgroup the wreath product $ \mathbf{H} = \mathcal{A}_{8} \bigotimes_{Wr} \mathbb{Z}_{3} $. The group ruling position changes only is $ \mathbf{C}= \mathbf{H}/\mathbf{T} $, where $ \mathbf{T} $ is the (normal) subgroup consisting of all al possible twists. It is folklore that $ \mathbf{C}\cong \mathcal{A}_{8} $ (see for example \cite{Bande82}).

Central cubies are distinguished according to the circle where they stand. In this particular case, we only have two circles, $k=1$ (the inner one) and $k=2$ (the outer one).
The inner circle of each face is formed by 4 cubies\footnote{They are actually the center cubies of the Rubik's Revenge.}: there is no need to split them into center corners and center edges as they are of of the same kind. We refer to $\mathbf{Z}_{c_1}$ as the subgroup of $\mathbf{G}_6$ permuting these 24 center cubies only. On the other hand, central cubies living in the second circle (for $k=2$) shall be divided into 24 center corners and 48 center edges, according to the convention introduced before. We refer to $\mathbf{Z}_{c_2}$ and $\mathbf{Z}_e$, as the subgroups permuting center corners standing in the second circle and center edges, respectively.

\begin{theorem}\label{th: Z_1 = A24}
$\mathbf{Z}_{c_1}\cong\mathbf{Z}_{c_2}\cong\mathcal{A}_{24}$.
\end{theorem}
\proof
We first show that $\mathbf{Z}_{c_1}\cong\mathcal{A}_{24}$. Then, similarly, we will prove that $\mathbf{Z}_{c_2}\cong\mathcal{A}_{24}$.
The group $ \mathbf{Z}_{c_1} $ is exactly the group permuting centers (only) in the Rubik's Revenge. The move $ z_1= [[C_{F},C_{D}],U^{-1}] $ gives a 3-cycle of any 3 arbitrary center corners standing in the inner circle, and this is enough to conclude that $ \mathcal{A}_{24}\leq \mathbf{Z}_{c_1} $. On the other hand, we claim that there exists no odd permutation in $ \mathbf{Z}_{c_1} $. Observe indeed that the only moves provoking an odd permutation on centers (in the inner circle) are $R,L,U,D,F, B$ (rotations of the most external slices). Indeed the set of moves $C_{f,2}$ (with $f\in\{R,L,U,D,F,B\}$) does not affect inner centers, while any of the move $C_{f}$ acts on the considered pieces as an even permutation (two cycles of 4 elements). All the moves $R,L,U,D,F,B$ produce odd permutations on corners, which cannot be restored after the action of such moves, as $\mathbf{C}\cong\mathcal{A}_8$. Hence, $ \mathbf{Z}_{c_1} $ possesses no odd permutation.
\noindent
The same strategy can be adapted in order to prove that $\mathbf{Z}_{c_2}\cong\mathcal{A}_{24}$. Indeed the move $ z_2= [[C_{F,2},C_{D,2}],U^{-1}] $ gives a 3-cycle of any 3 arbitrary center corners standing in the second circle (the most extenal in this particular case), and this gives that $ \mathcal{A}_{24}\leq \mathbf{Z}_{c_2} $. In order to show our claim, we only need to observe that no odd permutation can belong to $\mathbf{Z}_{c_2} $. Indeed, the only moves inducing an odd permutation are rotations $R,L,U,D,F,B$ (as moves $C_{f,2}$ act as even permutations on the considered pieces), but, as they also induce odd permutation on corners which cannot be restored afterwards, they cannot generate an odd permutation in $\mathbf{Z}_{c_2} $, hence $\mathbf{Z}_{c_2}\cong\mathcal{A}_{24}$.
\endproof

\begin{theorem}\label{Ze=S48}
$ \mathbf{Z}_{e} \cong \mathcal{A}_{48} $.
\end{theorem}
\begin{proof}
The move $ p= [[C_{F},C_{D,2}],U^{-1}] $ gives a 3-cycle of any three arbitrary center edges standing in the second circle, therefore $ \mathcal{A}_{48}\leq \mathbf{Z}_{e} $.

To show that no odd permutation belongs to $\mathbf{Z}_{e}$, it is enough to observe that all the (basic) moves shifting center edges permutes them of an even permutation (two 4-cycles). Therefore $ \mathbf{Z}_{e} \cong \mathcal{A}_{48} $.
\end{proof}

The number of (coupled) edges in the considered case is equal to 48 (24 inner and 24 outer). We refer to $ \mathbf{E}_1 $, $\mathbf{E}_2$, respectively, as the subgroups permuting the inner and outer, respectively, edges only. 

\begin{theorem}\label{th: E e S_e}
$ \mathbf{E}_1 \cong\mathbf{E}_2\cong S_{24}$.
\end{theorem}
\begin{proof}
The case $\mathbf{E}_1\cong S_{24}$ is proved as follows. 
The move $ e_1 = [C_L^{-1},[L,U^{-1}]] $ gives a 3-cycle of the most internal edge cubies\footnote{This is exactly the same move used for the Rubik's Revenge (see \cite[Theorem 4.3]{Loi}).} as depicted in Figure \ref{fig: triciclo E1}.

\begin{center}
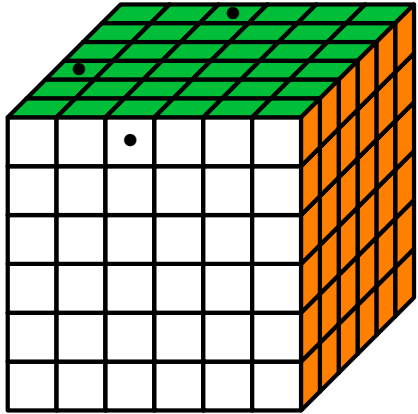
\begin{figure}[h]\label{fig:k pari}

\RubikCubeSolved
\begin{minipage}{10cm}
\centering


\begin{tikzpicture}[scale=0.65]
\DrawNCubeAll{6}{O}{G}{W}

\node at (0.2, 3.2) {$\bullet$};

\node at (-0.85, 4.65) {$\bullet$};

\node at (2.3, 5.8) {$\bullet$};

\end{tikzpicture}
\end{minipage}

\caption{The 3 cycle induced by the move $e_1$ on the most internal pair of edges.}\label{fig: triciclo E1}
\end{figure}
\end{center}

%
%
%
%
%
%
%
%
%
%

Therefore we have $ \mathcal{A}_{24}\leq \mathbf{E}_1 $. To show the other inclusion, we preliminary observe that the move 
$$ m = C_{R}^{2}B^{2}D^{2}C_{L}^{-1}D^{2}C_{R}D^{2}C_{R}^{-1}D^{2}F^{2}C_{R}^{-1}F^{2}C_{L}B^{2}C_{R}^{2} $$ 
gives a transposition of a pair of edges\footnote{The move may seem complicated at first sight, however it is well known by any cubemaster as it is an essential one to solve the Rubik's Revenge (for a description see for instance http://www.baumfamily.org/dave/rubiks.html).}. An easy calculation reveals that $ m $ acts as an even permutation on central cubies and corners. Therefore, by Theorems \ref{th: Z_1 = A24} and \ref{Ze=S48} there exists moves $\varphi\in\mathbf{Z}_{c_1}\times \mathbf{Z}_{c_2}$, $\psi\in\mathbf{Z}_e$ and $\chi\in\mathbf{C}$, s.t. $\varphi\circ\psi\circ\chi\circ m $ belongs to $\mathbf{E}_1$ and gives the desired transposition of edges (the composition $\varphi\circ\psi\circ\chi$ basically restores the ``mess'' provoked by $m$ on the other cubies). \\
\noindent
The analogous result for $\mathbf{E}_2$ can be proved similarly using the move $ e_2 = [C_{L,2}^{-1},[L,U^{-1}]] $ which produces any 3-cycle on the external edges (Figure \ref{fig: triciclo E2}) and the move $n = C_{R,2}^{2}B^{2}D^{2}$ \\ $C_{L,2}^{-1}D^{2}C_{R,2}D^{2} 
C_{R,2}^{-1}D^{2}F^{2}C_{R,2}^{-1}F^{2}C_{L,2}B^{2}C_{R,2}^{2} $ gives an odd permutation in $\mathbf{E}_2$.  
\begin{center}
\begin{figure}[h]

\RubikCubeSolved
\begin{minipage}{10cm}
\centering


\begin{tikzpicture}[scale=0.65]
\DrawNCubeAll{6}{O}{G}{W}

\node at (-0.8, 3.2) {$\bullet$};

\node at (-1.3, 4.25) {$\bullet$};

\node at (1.3, 5.8) {$\bullet$};

\end{tikzpicture}
\end{minipage}

\caption{The 3 cycle induced by the move $e_2$ on the most external pair of edges.}\label{fig: triciclo E2}
\end{figure}
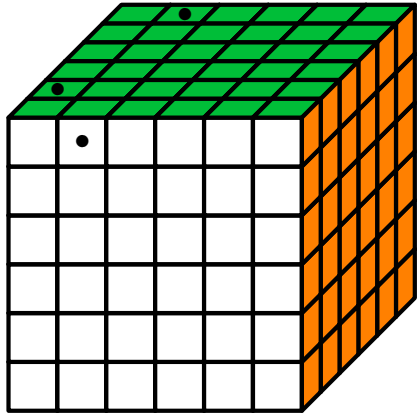
\end{center}

\end{proof}


\begin{remark}\label{rem: l'n si generalizza dal 6}
The strategies adopted to prove Theorems \ref{th: Z_1 = A24}, \ref{Ze=S48} and \ref{th: E e S_e} can be easily extended to the arbitrary $n\times n \times n$ cube by defining the appropriate set of moves inducing 3-cycles on any edge, center corners and center edge. 
\end{remark}
\subsection{Proof of Theorem \ref{th: first law per il cubo pari} for n=6.}\label{subsec: proof caso n=6}

Notice that a configuration in this case looks like $ (\sigma, \tau_1, \tau_2, \rho_{c_{1}}, \rho_{c_2}, \rho_{e_{2}}, x, y_1 , y_2) $ and it is valid if and only if  
\begin{enumerate}
\item $ sgn (\sigma)=sgn (\rho_{c_{1}})=sgn (\rho_{c_{2}}), $
\item $ \sum_{i} x_{i}\equiv 0 $ mod 3,
\item $y_{k_{ j_{t, s}}} = 1 - \delta_{t,s} $,

where $\delta_{a,a}=\delta_{b,b}=1$,  $\delta_{a,b}=\delta_{b,a}=0$, with $j=1,\dots ,12$ and $k=1,2$,
\item  $sgn(\rho_{e_{2}})= +1$.
\end{enumerate}

\begin{proof}
The first implication is proved by checking that all the basic moves generating $ \mathbf{M}_6 $, and thus $ \mathbf{G}_6 $, preserve conditions (1), (2), (3) and (4). The generators of $ \mathbf{M}_6 $ are $ R,L, U, D, F, B $ plus $ C_{R,k},C_{L,k}, C_{U,k}, C_{D,k}, C_{F,k}, C_{B,k} $, with $k\in\{ 1,2\}$.

Condition (1) affects only corners and central corners, so it is enough to check moves permuting those pieces. In particular $ R,L, U, D, F, B $, act as odd permutations on both corners and center corners; $ C_{f} $ with $ f\in \{ R,L, U, D, F, B \} $ act as identity on corners and center corners in the outer circle ($k=2$) and induce two 4-cycles on center corners in the first circle; similarly, moves $ C_{f,2} $ with $ f\in \{ R,L, U, D, F, B \} $ act as identity on corners and center corners in the first circle ($k=1$) and as even permutations (two 4-cycles) on the center corners in the second circle. Therefore, in any case, $ sgn (\sigma)=sgn (\rho_{c_{1}})=sgn (\rho_{c_{2}}) $. 

Condition (2) concerns only corners' orientations, hence its preservation is proved exactly the same way as for the Rubik's Cube, see \cite{Bande82}. 

Pairs of edges are permuted by the moves $ C_{R,k},C_{L,k}, C_{U,k}, C_{D,k}, C_{F,k}, C_{B,k} $. It is immediately checked that any $ C_{f,k} $, with $k\in\{ 1,2\}$, preserves condition (3). \\
\noindent
As for condition (4), all the moves involving center edges permute them of an even permutation, since all of them consist of two pairs of 4-cycles.

For the converse, suppose the Cube is in a configuration $ (\sigma, \tau_1 , \tau_2 ,  \rho_{c_{1}}, \rho_{c_2}, \rho_{e_{2}}, x, y_1 , y_2) $, satisfying conditions (1), (2), (3) and (4). 

 By condition (1), $ sgn(\sigma ) = sgn (\rho_{c_{1}})=sgn (\rho_{c_{2}}) $; this implies that either $ \sigma $, $ \rho_{c_{1}} $ and $\rho_{c_{2}}$ are all odd or even permutations. If all signs are odd, it is enough to apply for example $ R $ (it is not the only possibility), with the effect of getting $ sgn(\sigma ) = sgn (\rho_{c_{1}}) = sgn (\rho_{c_{2}})= +1 $. Having $ sgn(\sigma ) = sgn (\rho_{c_{1}}) =sgn (\rho_{c_{2}})= +1 $ we can apply Theorems \ref{C=A8} and \ref{th: Z_1 = A24}, which imply the existence of moves $ c\in \mathbf{C} $, $ z_1\in\mathbf{Z}_{c_1} $ and $ z_2\in\mathbf{Z}_{c_2} $ such that 
 $$ (c\circ z_1\circ z_2)\cdot(\sigma, \tau_{1},\tau_{2}, \rho_{c_{1}}, \rho_{c_2}, \rho_{e_{2}}, x, y_1 , y_2) = (Id, \tau_1,\tau_2, Id, Id, \rho_{e_{2}}, x, y_1 , y_2).   $$ 
\noindent
By Theorem \ref{th: E e S_e}, $ \mathbf{E}= \mathbf{E}_1\times\mathbf{E}_2\cong S_{24}\times S_{24} $, hence there exists a move $ e = (e_{1}, e_{2})\in \mathbf{E} $ such that 
$$ e\cdot (Id, \tau_1 , \tau_2 , Id, Id, \rho_{e_{2}}, x, y_1 , y_2)  = (Id, Id,Id,Id, Id, \rho_{e_{2}}, x, y_1 , y_2). $$
\noindent
Since by Condition (4), $sgn(\rho_{e_{2}})=+1$, applying Theorem \ref{Ze=S48} we can find a move $q\in \mathbf{Z}_e$ such that
 $$ q\cdot (Id, Id,Id, Id, \rho_{e_{2}}, x, y_1 , y_2) = (Id, Id,Id, Id, Id, x, y_1 , y_2).$$  
\noindent
The cube has been brought to a configuration where all cubies are correctly positioned, altought not correctly oriented. It follows from condition (3) that, whenever edges are correctly positioned, then they are also correctly oriented, hence $ y_1 = y_2 = 0 $. It only remains to orient corners correctly, but in the configuration $ (Id, Id, Id, Id, x, 0,0)  $ the Cube has been reduced to a big Rubik's Cube and, since condition (2) is fulfilled, corners can always be correctly oriented, see \cite{Bande82}. Therefore the initial configuration stands in the orbit of $ (\sigma, \tau_1,\tau_2, \rho_{c_{1}}, \rho_{c_2}, \rho_{e_{2}}, x, y_1 , y_2) $, hence the latter is valid.
\end{proof}

The reader may notice that the above displayed proof can be generalized to prove the first law of cubology for the `even cube' (Theorem \ref{th: first law per il cubo pari}) once the results on the subgroups of the Cube have been also proved (see Remark \ref{rem: l'n si generalizza dal 6}).

\subsection{The odd $n\times n\times n $ Rubik's Cube}\label{subsec: n dispari}

As any other cube, the number of corner cubies in the `odd cube' is 8, while the number of any other kind of piece (edges and centers) depends on $ n $. As previously observed, the most internal center cubie is \textit{fixed} in every face. The total numbers of center cubies and edges are the same introduced in \eqref{eq: numero dei centri caso pari} and in \eqref{eq: numero di edge}, respectively.

Notice that for $ n= 3 $, i.e. the Rubik's Cube, we have $ c = 6 $ and $ e = 12 $. 

In the `odd cube', exactly as in the particular case of Professor's Cube, edges can be splitted into single edges and coupled edges: the former are those 12 edges living in the most internal circle (in correspondence of the fixed center), while the latter are the edges living in the remaining circles. 

The positions and the orientations of corners can be mathematically described by a permutation $ \sigma\in S_{8} $ and a vector $ x\in (\mathbb{Z}_{3})^{8} $, respectively.

A permutation in $ S_{c} $ - with $ c $ defined in \eqref{eq: numero dei centri caso pari} is not the best candidate to describe the position of centers, as it does not take into account the existing difference between moving centers and fixed ones. Furthermore, a more refined description of center cubies shall be based on concentric circles, similarly as for the `even case'. 
Recall that by $ k = 0 $ we indicate the most internal circle, namely the degenerate circle consisting of the fixed center of each face only. 
It happens then that cubies standing in a circle, with exception of the most internal one ($ k= 0 $), can be divided into center corners and center edges: the former are the four cubies (in each circle) situated on the diagonals (with respect to the most internal circle) and the latter are those standing close to the side of the most internal circle. 
Adopting the same nomenclature of the ``even case'', we refer to $K$ as the total number of circles: in this case, $K=\frac{n-3}{2}$
Center edges may never take a (spatial) position occupied by center corners and vice versa. 
Notice that, for any circle, the number of center corners is equal to 24; hence the position assumed by any of them is described by using a permutation $ \rho_{c_{k}}\in S_{24} $, where the index $1\leq k\leq K$ refers to the circle (for example, $ \rho_{c_{1}} $ describes the position of the cubies in the most internal circle). On the other hand, the total number of center edge cubies is equal to $24K^2$ and the number $ z_k $ of center edges living in the $ k $-th circle, $k\geq 1$, is $24(2k-1)$. 

For this reason, a permutation $ \rho_{e_{k}}\in S_{z_{k}} $ describes the position assumed by center edge cubies. 

Established $ e $ being the number of edges, a permutation in $ S_{e} $ is not refined enough to describe positions of edges, as it does not take into account the distinction between single and coupled edges. Clearly, the positions of single edges may be encoded by $ \tau\in S_{12} $ (as they are only twelve). Any remaining edge is part of a pair (coupled edge) and can be moved by one among $ C_{R,k} $, $ C_{L,k} $, $ C_{U,k} $, $ C_{D,k} $, $ C_{F,k} $, $ C_{B,k} $. For instance, the companion of an edge moved by  $ C_{R,k} $ ($ C_{U,k} $, $ C_{F,k} $, respectively) is moved by $ C_{L,k} $ ($ C_{D,k} $, $ C_{B,k} $ respectively). For any $ k $, the number of (coupled) edges is equal to 24. Therefore, the position of any coupled edge can be described introducing $ \tau_k\in S_{24} $, with $k$ ranging over the number of circles. Furthermore, in any pair of edges one is of type $ a $ and its companion is of type $ b $, adopting, for each $ k $, the convention introduced for the Professor's Cube and for the `even cube'. Clearly also spatial position are divided into type $ a $ and $ b $, remembering the concept of type is used both to describe a property of (edge) cubies and of (spatial) positions.

Also orientation's vectors shall be differentiated for single and coupled edges: a vector $ z\in(\mathbb{Z}_{2})^{12} $ works for single edges, while vectors $ y_{k}\in(\mathbb{Z}_{2})^{24} $, for each $ k $, represent orientation of coupled edges in any circle\footnote{As for the ``even'' Cube, we adopt here the same terminology for edge cubies, namely their spatial position is described by a number $k_{j_t}$ (see Subsection \ref{subsec: n pari} for details).}. 

A configuration for the $ n\times n$ `odd' Rubik's Cube is a tuple $ (\sigma, \tau, \tau_{_k}, \rho_{c_{k}}, \rho_{e_{k}}, x, y_{_k},z) $, with $ \sigma\in S_{8} $, $ \tau\in S_{12} $, $ \tau_{_k}\in S_{24} $, $ \rho_{c_{k}}\in S_{24} $, $ \rho_{e_{k}}\in S_{z_{k}} $, for each $ k $ and $ x\in (\mathbb{Z}_{3})^{8} $, $ y_{_k}\in(\mathbb{Z}_{2})^{24} $, $ z\in(\mathbb{Z}_2)^{12} $.

The initial configuration is the tuple where all permutations are the respective identities and any vector's component is equal to 0. A configuration is valid if and only if it stands in the orbit of the initial configuration, i.e. if it might be obtained by the application of a finite sequence of moves from the initial configuration.  \\
The `first law of Cubology' for the odd Cube is a substantial generalization of the claim for the Professor's Cube, indeed:
\begin{theorem}\label{th: first law per il cubo dispari}
A configuration $ (\sigma, \tau, \tau_{_k}, \rho_{c_{k}}, \rho_{e_{k}}, x, y_{_k},z) $ of the `odd' Rubik's Cube is valid if and only if: 
\begin{enumerate}
\item[\emph{1.}] $ sgn (\sigma) = sgn (\tau) = sgn (\rho_{c_k}) $, for any $ k $, $1\leq k\leq K$,
\item[\emph{2.}] $ sgn (\tau_k) = sgn (\sigma)sgn (\rho_{e_k}) $, for any $ k $, $1\leq k\leq K $,
\item[\emph{3.}] $ \sum_{i}x_i\equiv 0 $ \emph{mod 3}, 
\item[\emph{4.}] $ \sum_{i}z_{i}\equiv 0 $ \emph{mod 2}, 
\item[\emph{5.}] $y_{k_{ j_{t, s}}} = 1 - \delta_{t,s} $,
\end{enumerate}
where $\delta_{a,a}=\delta_{b,b}=1$,  $\delta_{a,b}=\delta_{b,a}=0$.
\end{theorem}

The reader may check that the strategy used to prove Theorem \ref{the: condizioni5x5}, i.e. the special case with $n=5$, may be easily extended to give a purely algebraic proof of Theorem \ref{th: first law per il cubo dispari}.

\section*{Acknowledgements}

The work of the first author is supported by project GBP202/12/G061 of the Czech Science Foundation.
The second author was  supported by Prin 2015 -- Real and Complex Manifolds; Geometry, Topology and Harmonic Analysis -- Italy and also by INdAM and GNSAGA - Gruppo Nazionale per le Strutture Algebriche, Geometriche e le loro Applicazioni. Finally, we thank Fabio Zuddas and an anonymous referee for their many valuable suggestions on previous drafts of the paper.


\begin{thebibliography}{10}

\bibitem{Bande82}
C.~Bandelow.
\newblock {\em Inside Rubik's Cube and Beyond}.
\newblock Birkh{\"a}user, 1982.

\bibitem{Bonziothesis}
S.~Bonzio.
\newblock Algebraic structures from quantum and fuzzy logics.
\newblock PhD Thesis, Universit\`a di Cagliari, 2016.


\bibitem{Loi}
S.~Bonzio, A.~Loi, and L.~Peruzzi.
\newblock The first law of cubology for the rubik's revenge.
\newblock {\em Mathematica Slovaca}, 67(3), 2017.

\bibitem{Chen}
J.~Chen.
\newblock Group theory and the rubik's cube.
\newblock Notes, 2004.

\bibitem{Czech2011}
B.~Czech, K.~Larjo, and M.~Rozali.
\newblock Black holes as rubik's cubes.
\newblock {\em Journal of High Energy Physics}, 2011(8):143, 2011.

\bibitem{Demaine2011}
E.~D. Demaine, M.~L. Demaine, S.~Eisenstat, A.~Lubiw, and A.~Winslow.
\newblock Algorithms for solving rubik's cubes.
\newblock In C.~Demetrescu and M.~M. Halld{\'o}rsson, editors, {\em Algorithms
  -- ESA: 19th Annual European Symposium}, 2011.

\bibitem{Diaconu13}
A.~Diaconu and K.~Loukhaoukha.
\newblock An improved secure image encryption algorithm based on rubik's cube
  principle and digital chaotic cipher.
\newblock {\em Mathematical Problems in Engineering}, 2013.

\bibitem{Signm82}
A.~Frey and D.~Singmaster.
\newblock {\em Handbook of cubik math}.
\newblock Enslow Publishers, 1982.

\bibitem{OnGods}
M.~Jones, B.~Shelton, and M.~Weaverdyck.
\newblock On god's number(s) for rubik's slide.
\newblock {\em The College Mathematics Journal}, 45(4):267--275, 2014.

\bibitem{Joyner08}
D.~Joyner.
\newblock {\em Adventures in Group Theory}.
\newblock The Johns Hopkins University Press, 2008.

\bibitem{Kosniowski81}
C.~Kosniowski.
\newblock {\em Conquer that Cube}.
\newblock Cambridge University Press, 1981.

\bibitem{Kunkle07}
D.~Kunkle and G.~Cooperman.
\newblock Twenty-six moves suffice for rubik's cube.
\newblock In {\em Proceedings of the 2007 International Symposium on Symbolic
  and Algebraic Computation}, pages 235--242, 2007.

\bibitem{Larsen85}
M.~E. Larsen.
\newblock Rubik's revenge: The group theoretical solution.
\newblock {\em The American Mathematical Monthly}, 92(6), 1985.

\bibitem{Lee}
C.~L. Lee and M.~C. Huang.
\newblock The rubik's cube problem revisited: a statistical thermodynamic
  approach.
\newblock {\em Eur. Phys. J. B}, 64(2):257--261, 2008.

\bibitem{miller2012}
J.~Miller.
\newblock Move-count means with cancellation and word selection problems in
  rubik's cube solution approaches.
\newblock PhD Thesis, Kent State University   2012.


\bibitem{Rokicki14}
T.~Rokicki.
\newblock Towards god's number for rubik's cube in the quarter-turn metric.
\newblock {\em The College Mathematics Journal}, 45(4):242--242, 2014.

\bibitem{Rokicki13}
T.~Rokicki, H.~Kociemba, M.~Davidson, and J.~Dethridge.
\newblock The diameter of the rubik's cube group is twenty.
\newblock {\em SIAM Journal on Discrete Mathematics}, 27(2):1082--1105, 2013.

\bibitem{Volte2013}
E.~Volte, J.~Patarin, and V.~Nachef.
\newblock Zero knowledge with rubik's cubes and non-abelian groups.
\newblock In M.~Abdalla, C.~Nita-Rotaru, and R.~Dahab, editors, {\em Cryptology
  and Network Security: 12th International Conference}, pages 74--91, 2013.

\end{thebibliography}

\end{document}